\documentclass[12pt]{amsart}
\usepackage{amsfonts}
\usepackage{amssymb}
\usepackage{amsmath}

\newtheorem{theorem}{Theorem}[section]

\newtheorem{proposition}{Proposition}[section]

\newtheorem{definition}{Definition}[section]

\numberwithin{equation}{section}
\def\({\left ( }
\def\){\right )}
\def\<{\left < }
\def\>{\right >}

\begin{document}
\title[Transversal Lightlike Submanifolds]{Transversal Lightlike
submanifolds of Golden Semi Riemannian Manifolds}
\author{Feyza Esra Erdo\u{g}an}
\address{Faculty of Education 1\\
Adiyaman University TURKEY}
\email{ferdogan@adiyaman.edu.tr}
\author{Cumali Y\i ld\i r\i m}
\curraddr{Department of Mathematics\\
Inonu University TURKEY}
\email{cumali.yildirim@inonu.edu.tr}
\author{Esra Karatas}
\address{Department of Mathematics\\
Inonu University TURKEY}
\email{esrameryem44@gmail.com}

\begin{abstract}
The Golden Ratio is fascinating topic that continually generated news ideas.
A Riemannian manifold endowed with a Golden Structure will be called a
Golden Riemannian manifold. The main purpose of the present paper is to
study the geometry of radical transversal lightlike submanifolds and
transversal lightlike submanifolds of Golden Semi-Riemannian manifolds. We
investigate the geometry of distributions and obtain necessary and
sufficient conditions for the induced connection on these manifolds to be
metric connection. 
\end{abstract}

\maketitle

\section{INTRODUCTION}

The Golden proportion, also called the Golden ratio, Divine ratio, Golden
section or Golden mean, has been well known since the time of Euclid. Many
objects alive in the natural world that possess pentagonal symmetry, such as
inflorescence of many flowers and prophylaxis objects have a numerical
description given by the Fibonacci numbers which are themselves based on the
Golden proportion. The Golden proportion has also been found in the
structure of musical compositions, in the ratios of harmonious sound
frequencies and in dimensions of the human body. From ancient times it has
played an important role in architecture and visual arts. \indent {}\newline
\noindent \line(1,0){100}\newline

{\footnotesize \textit{Key words and phrases}: Lightlike manifold, golden
semi Riemannian manifold, radical transversal lightlike submanifold, radical
screen transversal submanifold }\newline
{\footnotesize \textit{2010 Mathematics Subject Classification}. 53C15,
53C40, 53C50.}\newline
\newpage The Golden proportion and the Golden rectangle (which is spanned by
two sides in the Golden proportion) have been found in the harmonious
proportion of temples, churches, statues, paintings, pictures and fractals.
Golden structure was revealed by the golden proportion, which was
characterized by Johannes Kepler(1571-1630). The number $\phi $, which is
the real positive root of the equation%
\begin{equation*}
x^{2}-x-1=0,
\end{equation*}%
(hence, $\phi =\frac{1+\sqrt{5}}{2}\approx 1,618...$) is the golden
proportion.

The Golden Ratio is fascinating topic that continually generated news ideas.
Golden Riemannian manifolds were introduced by Grasmereanu and Hretcanu \cite%
{GH} using Golden ratio. The authors also studied invariant submanifolds of
Golden Riemannian manifold and obtained interesting result in \cite{GH2},%
\cite{GH3}. The integrability of such Golden structures was also
investigated by Gezer, Cengiz and Salimov in \cite{GCS}. Moreover, the
harmonious of maps between Golden Riemannian manifolds was studied in \cite%
{BM}. Furthermore, Erdo\u{g}an and Y\i ld\i r\i m in \cite{EC} studied
semi-invariant and totally umbilical semi invariant submanifolds of Golden
Riemann manifolds, respectively. A semi-Riemannian manifold endowed with a
Golden structure will be called a Golden semi-Riemannian manifold. Precisely
we can say that on (1,1)-tensor field $\breve{P}$ on a m-dimensional
semi-Riemannin manifold $(\breve{N},\breve{g})$ is Golden structure if it
satisfies the equation $\breve{P}^{2}=\breve{P}+I$, where $I$ is identity
map on $\breve{N}$. Furthermore, 
\begin{equation*}
\breve{g}(\breve{P}W,U)=\breve{g}(W,\breve{P}U),
\end{equation*}%
the semi-Riemannian metric is called $\breve{P}$-compatible and $(\breve{N},%
\breve{g},\breve{P})$ is named a Golden semi-Riemannian manifold\cite{M}.

The theory of lightlike submanifolds of semi-Riemannian manifolds has a
important place in differential geometry. Lightlike submanifolds of a
semi-Riemannian manifolds has been studied by Duggal-Bejancu\cite{DB} and
Duggal and \c{S}ahin \cite{DS}. Indeed, lightlike submanifolds appear in
general relativity as some smooth parts of horizons of the Kruskal and Keer
black holes \cite{N}. Many authors studied the lightlike submanifolds in
different space, for example \cite{FC},\cite{BC},\cite{SBE} and \cite{FBR}.
Additionally, in a resent study, studied lightlike hypersurfaces of Golden
semi-Riemannian manifolds\cite{NE}.

Considering above information, in this paper, we introduce transversal
lightlike submanifolds of Golden semi-Riemannian manifolds and studied their
geometry. The paper is organized as follows: In Section2, we give basic
information needed for this paper. In Section3 and Section4 , we introduce
Golden semi-Riemannian manifold a long with its subclasses ( radical
transversal and transversal lightlike submanifolds) and obtain some
characterizations. We investigate the geometry of distiributions and find
necessary and sufficient conditions for induced connection to be metric
connection. Finally, we give an example.

\section{PRELIMINARIES}

A submanifold $\acute{N}^{m}$ immersed in a semi-Riemannian manifold $(%
\breve{N}^{m+k},\breve{g})$ is called a lightlike submanifold if it admits a
degenerate metric $g$\ induced from $\breve{g}$ whose radical distribution
which is a semi-Riemannian complementary distribution of $RadT\acute{N}$ is
of rank $r$, where $1\leq r\leq m.$ $RadT\acute{N}=T\acute{N}\cap T\acute{N}%
^{\perp }$ , where 
\begin{equation}
T\acute{N}^{\perp }=\cup _{x\in \acute{N}}\left\{ u\in T_{x}\breve{N}\mid 
\breve{g}\left( u,v\right) =0,\forall v\in T_{x}\acute{N}\right\} .
\end{equation}%
Let $S(T\acute{N})$ be a screen distribution which is a Semi-Riemannian
complementary distribution of $RadT\acute{N}$ in $T\acute{N}$. i.e., $T%
\acute{N}=RadT\acute{N}\perp S(T\acute{N}).$

We consider a screen transversal vector bundle $S(T\acute{N}^{\bot }),$
which is a semi-Riemannian complementary vector bundle of $RadT\acute{N}$ in 
$T\acute{N}^{\bot }.$ Since, for any local basis $\left\{ \xi _{i}\right\} $
of $RadT\acute{N}$, there exists a lightlike transversal vector bundle $ltr(T%
\acute{N})$ locally spanned by $\left\{ N_{i}\right\} $ \cite{DB}. Let $tr(T%
\acute{N})$ be complementary ( but not orthogonal) vector bundle to $T\acute{%
N}$ in $T\breve{N}^{\perp }\mid _{\acute{N}}$. Then%
\begin{align*}
tr(T\acute{N})& =ltrT\acute{N}\bot S(T\acute{N}^{\bot }), \\
T\breve{N}& \mid _{\acute{N}}=S(T\acute{N})\bot \lbrack RadT\acute{N}\oplus
ltrT\acute{N}]\perp S(T\acute{N}^{\bot }).
\end{align*}%
Although $S(T\acute{N})$ is not unique, it is canonically isomorphic to the
factor vector bundle $T\acute{N}/RadT\acute{N}$ \cite{DB}. The following
result is important for this paper.

\begin{proposition}
\cite{DB}. The lightlike second fundamental forms of a lightlike submanifold 
$\acute{N}$ do not depend on $S(T\acute{N}),$ $S(T\acute{N}^{\perp })$ and $%
ltr(T\acute{N}).$
\end{proposition}

We say that a submanifold $(\acute{N},g,S(T\acute{N}),$ $S(T\acute{N}^{\perp
}))$ of $\breve{N}$ is

Case1: \ r-lightlike if $r<\min\{m,k\};$

Case2: \ Co-isotropic if $r=k<m;S(T\acute{N}^{\perp})=\{0\};$

Case3: \ Isotropic if $r=m=k;$ $S(T\acute{N})=\{0\};$

Case4: \ Totally lightlike if $r=k=m;$ $S(T\acute{N})=\{0\}=S(T\acute {N}%
^{\perp}).$

\bigskip

The Gauss and Weingarten equations are:%
\begin{align}
\breve{\nabla}_{W}U& =\nabla _{W}U+h\left( W,U\right) ,\quad \forall W,U\in
\Gamma (T\acute{N}),  \label{7} \\
\breve{\nabla}_{W}V& =-A_{V}W+\nabla _{W}^{t}V,\quad \forall W\in \Gamma (T%
\acute{N}),\,V\in \Gamma (tr(T\acute{N})),  \label{8}
\end{align}%
where $\left\{ \nabla _{W}U,A_{V}W\right\} $ and $\left\{ h\left( W,U\right)
,\nabla _{W}^{t}V\right\} $ belong to $\Gamma (T\acute{N})$ and $\Gamma (tr(T%
\acute{N})),$ respectively. $\nabla $ and $\nabla ^{t}$ are linear
connections on $\acute{N}$ and the vector bundle $tr(T\acute{N})$,
respectively. Moreover, we have%
\begin{align}
\breve{\nabla}_{W}U& =\nabla _{W}U+h^{\ell }\left( W,U\right) +h^{s}\left(
W,U\right) ,\quad \forall W,U\in \Gamma (T\acute{N}),  \label{9} \\
\breve{\nabla}_{W}N& =-A_{N}W+\nabla _{W}^{\ell }N+D^{s}\left( W,N\right)
,\quad N\in \Gamma (ltr(T\acute{N})),  \label{10} \\
\breve{\nabla}_{W}Z& =-A_{Z}W+\nabla _{W}^{s}Z+D^{\ell }\left( W,Z\right)
,\quad Z\in \Gamma (S(T\acute{N}^{\bot })).  \label{11}
\end{align}%
Denote the projection of $T\acute{N}$ on $S(T\acute{N})$ by $\breve{P}.$Then
by using (\ref{7}), (\ref{9})-(\ref{11}) and a metric connection $\breve{%
\nabla}$, we obtain%
\begin{align}
\bar{g}(h^{s}\left( W,U\right) ,Z)+\bar{g}\left( U,D^{\ell }\left(
W,Z\right) \right) & =g\left( A_{Z}W,U\right) ,  \label{12} \\
\bar{g}\left( D^{s}\left( W,N\right) ,Z\right) & =\bar{g}\left(
N,A_{Z}W\right) .  \label{13}
\end{align}%
From the decomposition of the tangent bundle of a lightlike submanifold, we
have%
\begin{align}
\nabla _{W}\breve{P}U& =\nabla _{W}^{\ast }\breve{P}U+h^{\ast }(W,\breve{P}%
U),  \label{14} \\
\nabla _{W}\xi & =-A_{\xi }^{\ast }W+\nabla _{W}^{\ast t}\xi ,  \label{15}
\end{align}%
for $W,U\in \Gamma (T\acute{N})$ and $\xi \in \Gamma (RadT\acute{N}).$ By
using above equations, we obtain%
\begin{align}
g(h^{\ell }(W,\breve{P}U),\xi )& =g(A_{\xi }^{\ast }W,\breve{P}U),
\label{16} \\
g(h^{s}(W,\breve{P}U),N)& =g(A_{N}W,\breve{P}U),  \label{17} \\
g\left( h^{\ell }\left( W,\xi \right) ,\xi \right) & =0,A_{\xi }^{\ast }\xi
=0.  \label{18}
\end{align}%
In general, the induced connection $\nabla $ on $\acute{N}$ is not a metric
connection. Since $\breve{\nabla}$ is a metric connection, by using (\ref{9}%
) we get%
\begin{equation}
\left( \nabla _{W}g\right) \left( U,V\right) =\breve{g}\left( h^{\ell
}\left( W,U\right) ,V\right) +\breve{g}\left( h^{\ell }\left( W,V\right)
,U\right) .  \label{19}
\end{equation}%
However, to pay attention that $\nabla ^{\ast }$ is a metric connection on $%
S(T\acute{N})$.

\bigskip

Let $(N_{1},g_{1})$ and $(N_{2},g_{2})$ be two $m_{1}$ and $m_{2}$
-dimensional semi-Riemannian manifolds with constant indexes $q_{1}>0,$ $%
q_{2}>0,$ respectively. Let $\pi:N_{1}\times N_{2}\rightarrow N_{1}$ and $%
\sigma:N_{1}\times N_{2}\rightarrow N_{2}$ the projections which are given
by $\pi(w,u)=w$ and $\sigma(w,u)=u$ for $\ $any $(w,u)\in N_{1}\times N_{2},$
respectively.

We denote the product manifold by $\breve{N}=(N_{1}\times N_{2},\breve{g}),$
where%
\begin{equation*}
\breve{g}(W,U)=g_{1}(\pi _{\ast }W,\pi _{\ast }U)+g_{2}(\sigma _{\ast
}W,\sigma _{\ast }U)
\end{equation*}%
for any $\forall W,U\in \Gamma (T\breve{N}).$ Then we have%
\begin{equation*}
\pi _{\ast }^{2}=\pi _{\ast },\quad \pi _{\ast }\sigma _{\ast }=\sigma
_{\ast }\pi _{\ast }=0,
\end{equation*}%
\begin{equation*}
\sigma _{\ast }^{2}=\sigma _{\ast },\quad \pi _{\ast }+\sigma _{\ast }=I,
\end{equation*}%
where $I$ is identity transformation. Thus $(\breve{N},\breve{g})$ is an $%
\left( m_{1}+m_{2}\right) $- dimensional semi-Riemannian manifold with
constant index $\left( q_{1}+q_{2}\right) .$ The semi-Riemannian product
manifold $\breve{N}=N_{1}\times N_{2}$ is characterized by $N_{1}$ and $%
N_{2} $ are totally geodesic submanifolds of $\breve{N}.$

Now, if we put $F=\pi _{\ast }-\sigma _{\ast }$, then we can easily see that 
\begin{align*}
F.F& =\left( \pi _{\ast }-\sigma _{\ast }\right) \left( \pi _{\ast }-\sigma
_{\ast }\right) \\
F^{2}& =\pi _{\ast }^{2}-\pi _{\ast }\sigma _{\ast }-\sigma _{\ast }\pi
_{\ast }+\sigma _{\ast }^{2}=I
\end{align*}%
\begin{equation*}
F^{2}=I,\quad \breve{g}\left( FW,U\right) =\breve{g}\left( W,FU\right)
\end{equation*}%
for any $W,U\in \Gamma (T\breve{N}).$ If we denote the Levi-Civita
connection on $\breve{N}$ by $\breve{\nabla},$ then it can be seen that $(%
\breve{\nabla}_{W}F)Y=0,$ for any $W,U\in \Gamma (T\breve{N}),$ that is, $F$
is parallel with respect to $\breve{\nabla}$ \cite{S}$.$

Let $(\breve{N},\breve{g})$ be a semi-Riemannian manifold. Then $\breve{N}$
is called Golden semi-Riemannian manifold if there exists an $(1,1)$ tensor
field $\breve{P}$ on $\breve{N}$ such that{\small \ }%
\begin{equation}
\breve{P}^{2}=\breve{P}+I  \label{20}
\end{equation}%
where $I$ is the identity map on $\breve{N}.$ Also, 
\begin{equation}
\breve{g}(\breve{P}W,U)=\breve{g}(W,\breve{P}U).  \label{21}
\end{equation}%
\newline
The semi-Riemannian metric (\ref{21}) is called $\breve{P}-$compatible and $(%
\breve{N},\breve{g},\breve{P})$ is named a Golden semi-Riemannian manifold.
Also a Golden semi-Riemannian structure $\breve{P}$ is called a locally
Golden structure if $\breve{P}$ is parallel with respect to the Levi-Civita
connection $\breve{\nabla},$ that is 
\begin{equation}
\breve{\nabla}_{W}\breve{P}U=\breve{P}\breve{\nabla}_{W}U  \label{21a}
\end{equation}%
\cite{CH}.

If $\breve{P}$ be a Golden structure, then (\ref{21}) is equivalent to 
\begin{equation}
\breve{g}(\breve{P}W,\breve{P}U)=\breve{g}(\breve{P}W,U)+\breve{g}(W,U)
\label{23}
\end{equation}%
for any $W,U\in \Gamma (T\breve{N}).$

If $F$ is almost product structure on $\breve{N}$, then%
\begin{equation}
\breve{P}=\frac{1}{\sqrt{2}}(I+\sqrt{5}F)  \label{24}
\end{equation}%
is a Golden structure on $\breve{N}.$ Conversely, if $\breve{P}$ is a Golden
structure on $\breve{N}$, then 
\begin{equation*}
F=\frac{1}{\sqrt{5}}(2\breve{P}-I),
\end{equation*}%
is an almost product structure on $\breve{N}$\cite{N}.

\section{RADICAL TRANSVERSAL LIGHTLIKE SUBMANIFOLDS OF GOLDEN
SEMI-RIEMANNIAN MANIFOLDS}

\begin{definition}
Let $(\acute{N},\breve{g},S(T\acute{N}),S(T\acute{N})^{\perp })$ be
lightlike submanifold of a Golden semi-Riemannian manifold. If the following
conditions are provided, the lightlike submanifold called radical
transversal lightlike submanifold.%
\begin{equation}
\breve{P}RadT\acute{N}=ltr(T\acute{N})  \label{1}
\end{equation}%
\begin{equation}
\breve{P}(S(T\acute{N}))=S(T\acute{N})  \label{2}
\end{equation}
\end{definition}

\begin{proposition}
Let $\acute{N}$ be Golden semi-Riemannian manifold. In this case there is no
1-lightlike radical transversal lightlike submanifold of the $\breve{N}$
manifold.
\end{proposition}

\begin{proof}
Let's assume that the submanifold is $\acute{N}$ 1-lightlike. In this case,%
\begin{equation*}
RadT\acute{N}=Sp\left\{ \xi \right\} ,
\end{equation*}%
and 
\begin{equation*}
ltrT\acute{N}=Sp\left\{ N\right\} .
\end{equation*}%
Then 
\begin{equation}
g(\breve{P}W,\breve{P}U)=g(W,\breve{P}U)+g(W,U),  \label{*}
\end{equation}%
with the help of (\ref{*}) equation%
\begin{equation}
g(\breve{P}\xi ,\xi )=g(\xi ,\breve{P}\xi )=g(\breve{P}\xi ,\breve{P}\xi
)-g(\xi ,\xi )=0,  \label{3}
\end{equation}%
is obtained. On the other hand from (\ref{1}) equation 
\begin{equation*}
\breve{P}\xi =N\in ltr(T\acute{N}),
\end{equation*}%
must be. Then we obtain.%
\begin{equation*}
g(\xi ,\breve{P}\xi )=g(\breve{P}\xi ,\xi )=g(N,\xi )=g(\xi ,N)=1.
\end{equation*}%
This is contradictory to equation (\ref{3}).

In that case, our acceptance is wrong; the truth given to us in the
hypothesis.
\end{proof}

For $V\in\Gamma(S(T\acute{N}))$ and $\xi\in\Gamma(RadT\acute{N})$ from (\ref%
{*}) equation we have 
\begin{equation*}
g(\breve{P}V,\xi)=g(V,\breve{P}\xi)=0
\end{equation*}
This indicates that there is no component of $\breve{P}V$ vector field in $%
ltrT\acute{N}$. Similarly for $N\in\Gamma(ltrT\acute{N})$%
\begin{align*}
g(\breve{P}V,N) & =g(V,\breve{P}N) \\
& =g(\breve{P}V,\breve{P}N)-g(V,N) \\
& =g(\breve{P}V,\breve{P}N)
\end{align*}
then 
\begin{equation*}
g(\breve{P}V,N)=g(V,\breve{P}N)=g(\breve{P}V,\breve{P}N)
\end{equation*}
From definition of the Radical transversal lightlike submanifold for $\xi
\in\Gamma(RadT\acute{N})$ and $N_{1}\in\Gamma(ltrT\acute{N})$%
\begin{equation*}
\breve{P}\xi_{1}=N_{1}
\end{equation*}
If we apply $\breve{P}$ to this expression we have 
\begin{align*}
\breve{P}^{2}\xi_{1} & =\breve{P}N_{1} \\
\breve{P}\xi_{1}+\xi_{1} & =\breve{P}N_{1}
\end{align*}

Then we have 
\begin{equation*}
g(\breve{P}V,N)=g(V,\breve{P}N)=0
\end{equation*}

Then it is understood that there is no component of $\breve{P}V$ in $RadT%
\acute{N}$. For $W\in\Gamma(S(T\acute{N}))$ 
\begin{equation*}
g(\breve{P}V,W)=g(V,\breve{P}W)=0
\end{equation*}
The result is that the $\breve{P}V$ vector field has no component in the $S($
$T\acute{N})$. That is, $\breve{P}V$ is neither a component of $S($ $T\acute{%
N})$ nor of $RadT\acute{N}$ nor of $ltrT\acute{N}$, but of $\breve {P}V\in
S(T\acute{N})$. Here we come to the conclusion.

\begin{theorem}
Let $\breve{N}$ be locally Golden semi-Riemannian manifold and $\acute{N}$
be the radical transversal lightlike submanifold of $\breve{N}$. In this
case the $S(T\acute{N})$ distribution is invariant.
\end{theorem}

Let $\breve{N}$ be golden semi-Riemannian manifold and $\acute{N}$ be the
radical transversal lightlike submanifold of $\breve{N}$. The projection
transforms on $RadT\acute{N}$ and $S($ $T\acute{N})$ are respectively $Q$
and $T$ for $\forall W\in \Gamma (T\acute{N})$.%
\begin{equation}
W=TW+QW,  \label{4}
\end{equation}

where $TW\in \Gamma (S(T\acute{N}))$ and $QW\in \Gamma (RadT\acute{N})$

If $\breve{P}$ is applied to equation (\ref{4}), then we have 
\begin{equation}
\breve{P}W=\breve{P}TW+\breve{P}QW.  \label{5}
\end{equation}%
If $\breve{P}TW=SW$ and $\breve{P}QW=LW$, then equation (\ref{5}) turns into 
\begin{equation}
\breve{P}W=SW+LW,  \label{6}
\end{equation}%
where $SW\in \Gamma (S(T\acute{N}))$ and $LW\in \Gamma (ltrT\acute{N}).$

Let $\breve{N}$ be locally\ Golden semi-Riemannian manifold and $\acute{N}$
be radical transversal lightlike submanifold of $\breve{N}$. Because 
\begin{align*}
\breve{\nabla}_{U}\breve{P}W& =\breve{P}\breve{\nabla}_{U}W \\
\breve{\nabla}_{U}(SW+LW)& =\breve{P}(\breve{\nabla}%
_{U}W+h^{s}(U,W)+h^{l}(U,W) \\
\breve{\nabla}_{U}SW+\breve{\nabla}_{U}LW& =\breve{P}\breve{\nabla}_{U}W+%
\breve{P}h^{s}(U,W)+\breve{P}h^{l}(U,W)
\end{align*}

\begin{equation*}
\binom{\nabla _{U}SW+h^{s}(U,SW)+h^{l}(U,SW)}{-A_{LW}U+\nabla
_{U}^{l}LW+D^{s}(U,LW)}=\left( 
\begin{array}{c}
S\nabla _{U}W+L\nabla _{U}W \\ 
+\breve{P}h^{s}(U,W)+\breve{P}h^{l}(U,W)%
\end{array}%
\right) ,
\end{equation*}%
where $\breve{P}h^{l}(U,W)=K_{1}\breve{P}h^{l}(U,W)+K_{2}\breve{P}h^{l}(U,W)$
; also $K_{1}$ and $K_{2}$ projection morphisms of $\breve{P}ltrT\acute{N}$
on $ltrT\acute{N}$ and $RadT\acute{N}$ respectively,

\begin{equation*}
\binom{\nabla _{U}SW+h^{s}(U,SW)+h^{l}(U,SW)}{-A_{LW}U+\nabla
_{U}^{l}LW+D^{s}(U,LW)}=\left( 
\begin{array}{c}
S\nabla _{U}W+L\nabla _{U}W+\breve{P}h^{s}(U,W) \\ 
+K_{1}\breve{P}h^{l}(U,W)+K_{2}\breve{P}h^{l}(U,W)%
\end{array}%
\right) .
\end{equation*}%
If the tangent, screen transversal, lightlike transversal components of this
equation are written separately 
\begin{equation*}
\nabla _{U}SW-A_{LW}U=S\nabla _{U}W+K_{2}\breve{P}h^{l}(U,W),
\end{equation*}%
\begin{align*}
h^{s}(U,SW)+D^{s}(U,LW)& =\breve{P}h^{s}(U,W), \\
h^{l}(U,SW)+\nabla _{U}^{l}LW& =L\nabla _{U}W+K_{1}\breve{P}h^{l}(U,W),
\end{align*}%
from above equations 
\begin{align*}
\nabla _{U}SW-S\nabla _{U}W& =A_{LW}U+K_{2}\breve{P}h^{l}(U,W), \\
& \Rightarrow (\nabla _{U}S)W=A_{LW}U+K_{2}\breve{P}h^{l}(U,W)
\end{align*}%
\begin{align*}
h^{s}(U,SW)+D^{s}(U,LW)-\breve{P}h^{s}(U,W)& =0, \\
h^{l}(U,SW)+\nabla _{U}^{l}LW-L\nabla _{U}W-K_{1}\breve{P}h^{l}(U,W)& =0,
\end{align*}%
are obtained.Then we have this proposition.

\begin{proposition}
Let $\breve{N}$ be locally Golden semi-Riemannian manifold and $\acute{N}$
be radical transversal lightlike submanifold of $\breve{N}$. For $\forall
W,U\in \Gamma (T\acute{N})$ we have following statements:%
\begin{equation}
(\nabla _{U}S)W=A_{LW}U+K_{2}\breve{P}h^{l}(U,W)  \label{7}
\end{equation}%
\begin{equation}
h^{s}(U,SW)+D^{s}(U,LW)-\breve{P}h^{s}(U,W)=0  \label{8}
\end{equation}%
\begin{equation}
h^{l}(U,SW)+\nabla _{U}^{l}LW-L\nabla _{U}W-K_{1}\breve{P}h^{l}(U,W)=0
\label{9}
\end{equation}
\end{proposition}

Let's now give the theorem that contains the necessary and sufficient
conditions for the reduced connection to be metric.

\begin{theorem}
Let $\acute{N}$ be radical transversal lightlike submanifold of locally
Golden semi-Riemannian $\breve{N}$ manifold. Necessary and sufficient
condition for the induced connection $\nabla $ on $M$ is metric connection
is that for $W\in \Gamma (T\acute{N})$ and $\xi \in \Gamma (RadT\acute{N})$ 
\begin{equation*}
A_{\breve{P}\xi }W\in \Gamma (RadT\acute{N}).
\end{equation*}
\end{theorem}

\begin{proof}
Let's assume that $\nabla $ be metric connection. Then for $W\in \Gamma (T%
\acute{N})$ and $\xi \in \Gamma (RadT\acute{N})$ 
\begin{equation*}
\nabla _{W}\xi \in \Gamma (RadT\acute{N}),
\end{equation*}%
is obtained. Then for$\ U\in \Gamma (S(T\acute{N}))$%
\begin{align*}
g(\nabla _{W}\xi ,U)& =\breve{g}(\breve{\nabla}_{W}\xi -h^{l}(W,\xi
)-h^{s}(W,\xi ),U), \\
0& =\breve{g}(\breve{\nabla}_{W}\xi ,U).
\end{align*}%
On the other hand if the following equation is used 
\begin{align*}
\breve{g}(\breve{P}X,\breve{P}Y)& =\breve{g}(X,\breve{P}Y)+\breve{g}(X,Y), \\
0& =\breve{g}(\breve{P}\breve{\nabla}_{W}\xi ,\breve{P}U)-\breve{g}(\breve{%
\nabla}_{W}\xi ,\breve{P}U), \\
0& =\breve{g}(\breve{\nabla}_{W}\breve{P}\xi ,\breve{P}U)-\breve{g}(\breve{%
\nabla}_{W}\xi ,\breve{P}U), \\
0& =\breve{g}(-A_{\breve{P}\xi }W+\nabla _{W}^{l}\breve{P}\xi +D^{s}(W,%
\breve{P}\xi ),\breve{P}U)-\breve{g}(\nabla _{W}\xi ,\breve{P}U) \\
0& =-\breve{g}(A_{\breve{P}\xi }W,\breve{P}U)
\end{align*}%
is obtained. Where the result is that $A_{\breve{P}\xi }W$ is valued at$\
RadT\acute{N}$. Conversely assuming $A_{\breve{P}\xi }W$ is valued at $RadT%
\acute{N}$ for $U\in \Gamma (S(T\acute{N}))$%
\begin{align*}
g(A_{\breve{P}\xi }W,U)& =0, \\
g(-\breve{\nabla}_{W}\breve{P}\xi +\nabla _{W}^{l}\breve{P}\xi +D^{s}(W,%
\breve{P}\xi ),U)& =0, \\
-g(\breve{\nabla}_{W}P\xi ,U)& =0,
\end{align*}%
\begin{align*}
g(\breve{P}\breve{\nabla}_{W}\xi ,U)=0& \Rightarrow g(\breve{\nabla}_{W}\xi ,%
\breve{P}U)=0, \\
& \Rightarrow g(\nabla _{W}\xi +h^{l}(W,\xi )+h^{s}(W,\xi ),\breve{P}U)=0, \\
& \Rightarrow g(\nabla _{W}\xi ,\breve{P}U)=0, \\
& \Longrightarrow \nabla _{W}\xi \in \Gamma (RadT\acute{N})
\end{align*}%
is obtained.
\end{proof}

Now, let's examine the conditions for integrable distributions in the
direction of definition of radical transversal lightlike submanifolds.

\begin{theorem}
Let $\acute{N}$ be radical transversal lightlike submanifold of locally
Golden semi-Riemannian $\breve{N}$ manifold. In this case necessary and
sufficient condition for the distribution $S(T\acute{N})$ integrable is that
for $\forall W,U\in \Gamma (S(T\acute{N}))$%
\begin{equation*}
h^{l}(U,SW)=h^{l}(W,SU).
\end{equation*}
\end{theorem}

\begin{proof}
For $W,U\in \Gamma (S(T\acute{N}))$ roles of $U$ and $W$ are changed in the (%
\ref{9}) equation 
\begin{align*}
h^{l}(U,SW)+\nabla _{U}^{l}LW-L\nabla _{U}W-K_{1}\breve{P}h^{l}(U,W)& =0, \\
h^{l}(W,SU)+\nabla _{W}^{l}LU-L\nabla _{W}U-K_{1}\breve{P}h^{l}(W,U)& =0,
\end{align*}%
\begin{equation*}
\Rightarrow \left( 
\begin{array}{c}
h^{l}(U,SW)-h^{l}(W,SU)+\nabla _{U}^{l}LW \\ 
-\nabla _{W}^{l}LU-L\nabla _{U}W+L\nabla _{W}U- \\ 
K_{1}\breve{P}h^{l}(U,W)+K_{1}\breve{P}h^{l}(W,U)%
\end{array}%
\right) =0
\end{equation*}%
\begin{equation}
\Rightarrow \left( 
\begin{array}{c}
h^{l}(U,SW)-h^{l}(W,SU) \\ 
-K_{1}(\breve{P}h^{l}(U,W)-\breve{P}h^{l}(W,U))%
\end{array}%
\right) =L\left[ U,W\right]   \label{10}
\end{equation}%
is obtained. Because of $h^{l}$ symmetric we get 
\begin{equation}
h^{l}(U,SW)-h^{l}(W,SU)=L\left[ U,W\right]   \label{11}
\end{equation}%
Proof is obtained from equation (\ref{11}).
\end{proof}

\begin{theorem}
Let $\acute{N}$ be radical transversal lightlike submanifold of locally
Golden semi-Riemannian $\breve{N}$ manifold. In this case necessary and
sufficient condition for the radical distribution integrable is that for $%
\forall W,U\in \Gamma (RadT\acute{N})$%
\begin{equation*}
A_{LU}W=A_{LW}U.
\end{equation*}
\end{theorem}

\begin{proof}
For $W,U\in \Gamma (RadT\acute{N})$ taking into account of (\ref{7}) we get 
\begin{equation*}
(\nabla _{U}S)W=A_{LW}U+K_{2}\breve{P}h^{l}(U,W),
\end{equation*}%
where because of $W\in \Gamma (RadT\acute{N}),\quad SW=0.$

Thus we obtain 
\begin{align*}
\nabla _{U}SW-S\nabla _{U}W& =A_{LW}U+K_{2}\breve{P}h^{l}(U,W). \\
\Rightarrow -S\nabla _{U}W=& A_{LW}U+K_{2}\breve{P}h^{l}(U,W).
\end{align*}%
Here, if the roles of $U$ and $W$ are changed 
\begin{align*}
-S\nabla _{U}W& =A_{LW}U+K_{2}\breve{P}h^{l}(U,W), \\
-S\nabla _{W}U& =A_{LU}W+K_{2}\breve{P}h^{l}(W,U),
\end{align*}%
\begin{equation*}
\Rightarrow S(\nabla _{W}U-\nabla _{U}W)=\left( 
\begin{array}{c}
A_{LW}U-A_{LU}W \\ 
+K_{2}\breve{P}(h^{l}(U,W)-h^{l}(W,U))%
\end{array}%
\right) ,
\end{equation*}%
also, given that $h^{l}$ symmetric we get 
\begin{equation*}
S\left[ W,U\right] =A_{LW}U-A_{LU}W.
\end{equation*}%
Thus the proof is complete.
\end{proof}

\begin{theorem}
Let $\acute{N}$ be radical transversal lightlike submanifold of locally
Golden semi-Riemannian $\breve{N}$ manifold. In this case, necessary and
sufficient condition for radical distribution definition of totally geodesic
foliation on $\acute{N}$ is%
\begin{equation*}
h^{\ast }(W,\breve{P}Z)=h^{\ast }(W,Z)
\end{equation*}%
for $W\in \Gamma (RadT\acute{N})$, $Z\in \Gamma (S(T\acute{N}))$.
\end{theorem}

\begin{proof}
Necessary and sufficient condition for $RadT\acute{N}$ distribution
definition of totally geodesic foliation is 
\begin{equation*}
g(\nabla _{W}U,Z)=0,
\end{equation*}%
for $W$, $U\in \Gamma (RadT\acute{N})$, $Z\in \Gamma (S(T\acute{N}))$. Where 
\begin{equation*}
g(\nabla _{W}U,Z)=\breve{g}(\breve{\nabla}_{W}U,Z).
\end{equation*}%
Since $\breve{\nabla}$ is a metric connection 
\begin{equation*}
g(\nabla _{W}U,Z)=W\breve{g}(U,Z)-g(U,\breve{\nabla}_{W}Z),
\end{equation*}%
is obtained. If 
\begin{equation*}
g(\breve{P}X,\breve{P}Y)=g(X,\breve{P}Y)+g(X,Y),
\end{equation*}%
equation is used here for $\breve{P}Z\in S(T\acute{N})$, $Z\in S(T\acute{N})$%
, $W\in \Gamma (RadT\acute{N})$, $\breve{P}U\in \Gamma (ltrT\acute{N})$ 
\begin{align*}
g(\nabla _{W}U,Z)& =-\breve{g}(\breve{P}U,\breve{P}\breve{\nabla}_{W}Z)+%
\breve{g}(U,\breve{P}\breve{\nabla}_{W}Z) \\
& =-\breve{g}(\breve{P}U,\breve{\nabla}_{W}\breve{P}Z)+\breve{g}(\breve{P}U,%
\breve{\nabla}_{W}Z) \\
& =\breve{g}(\breve{P}U,-\breve{\nabla}_{W}\breve{P}Z+\breve{\nabla}_{W}Z) \\
& =\breve{g}(\breve{P}U,-\breve{\nabla}_{W}\breve{P}Z+\nabla _{W}Z) \\
& =\breve{g}(\breve{P}U,-\nabla _{W}^{\ast }PZ-h^{\ast }(W,\breve{P}%
Z)+\nabla _{W}^{\ast }Z+h^{\ast }(W,Z)) \\
& =\breve{g}(\breve{P}U,-h^{\ast }(W,\breve{P}Z)+h^{\ast }(W,Z)
\end{align*}%
it ends with proof.
\end{proof}

\begin{theorem}
Let $\acute{N}$ be radical transversal lightlike submanifold of locally
Golden semi-Riemannian $\breve{N}$ manifold. Then necessary and sufficient
condition for screen distribution defining \ totally geodesic foliation on $%
\acute{N}$ for $W,U\in \Gamma (S(T\acute{N}))$ and there is no component in $%
\Gamma (ltrT\acute{N})$ in $\breve{P}N$ \ for $N\in \Gamma (ltrT\acute{N})$
or 
\begin{equation*}
h^{\ast }(W,\breve{P}U)+K_{2}h^{l}(W,\breve{P}U)=h^{\ast
}(W,U)+K_{2}h^{l}(W,U)
\end{equation*}
\end{theorem}

\begin{proof}
Necessary and sufficient condition for $S(T\acute{N})$ defining \ totally
geodesic foliation 
\begin{equation*}
g(\nabla_{W}U,N)=0
\end{equation*}
for $W,U\in\Gamma(S(T\acute{N}))$, $N\in\Gamma(ltrT\acute{N})$

From here for $\breve{P}N\in ltrT\acute{N}+RadT\acute{N}$ and $\breve{P}U\in
S(T\acute{N})$%
\begin{align*}
g(\nabla _{W}U,N)& =\breve{g}(\breve{\nabla}_{W}U-h^{l}(W,U)-h^{s}(W,U),N) \\
& =\breve{g}(\breve{\nabla}_{W}U,N) \\
& =\breve{g}(P\breve{\nabla}_{W}U,\breve{P}N)-\breve{g}(\breve{\nabla}_{W}U,%
\breve{P}N) \\
& =\breve{g}(\breve{\nabla}_{W}\breve{P}U,\breve{P}N)-\breve{g}(\breve{\nabla%
}_{W}U,\breve{P}N) \\
& =\breve{g}(\nabla _{W}\breve{P}U+h^{l}(W,\breve{P}U)+h^{s}(W,\breve{P}U),%
\breve{P}N) \\
& -\breve{g}(\nabla _{W}U+h^{l}(W,U)+h^{s}(W,U),\breve{P}N) \\
& =\breve{g}(\nabla _{W}^{\ast }\breve{P}U+h^{\ast }(W,\breve{P}U)+h^{l}(W,%
\breve{P}U)),\breve{P}N) \\
& -\breve{g}(\nabla _{W}^{\ast }U+h^{\ast }(W,U)+h^{l}(W,U),\breve{P}N) \\
\breve{g}(\nabla _{W}U,N)& =\breve{g}(h^{\ast }(W,\breve{P}U)+K_{2}h^{l}(W,%
\breve{P}U)-h^{\ast }(W,U)-K_{2}h^{l}(W,U),\breve{P}N) \\
0& =\breve{g}(h^{\ast }(W,\breve{P}U)+K_{2}h^{l}(W,\breve{P}U)-h^{\ast
}(W,U)-K_{2}h^{l}(W,U),\breve{P}N)
\end{align*}%
here is the possible outcome: There is no component in $\breve{P}N$ in $ltrT%
\acute{N}$ or 
\begin{equation*}
h^{\ast }(W,\breve{P}U)+K_{2}h^{l}(W,\breve{P}U)=h^{\ast
}(W,U)+K_{2}h^{l}(W,U)
\end{equation*}
\end{proof}

\section{TRANSVERSAL LIGHTLIKE SUBMANIFOLDS}

In this section, definition of transversal lightlike submanifolds is given
and the geometry of the distributions is examined.

\begin{definition}
\label{bir}Let $(\acute{N},g,S(T\acute{N}),S(T\acute{N}))$ be lightlike
submanifold of Golden semi-Riemannian manifold $(\breve{N},\breve{g})$. If
the following conditions are provided the submanifold is called a
transversal lightlike submanifold. 
\begin{align*}
\breve{P}(RadT\acute{N})& =ltrT\acute{N}, \\
\breve{P}(S(T\acute{N}))& \subseteq S(T\acute{N}).
\end{align*}%
We show with the subbundle $\mu $ which is orthogonal complement to the $%
\breve{P}(S(T\acute{N}))$ in the $S(T\acute{N}).$
\end{definition}

\begin{proposition}
Let $\acute{N}$ be, transversal lightlike submanifold of Golden
semi-Riemannian $\breve{N}$ manifold. In this case $\mu $ distribution is
invariant to $\breve{P}$.
\end{proposition}

\begin{proof}
For $V\in \Gamma (\mu )$, $\xi \in \Gamma (RadT\acute{N})$ and $N\in \Gamma
(ltrT\acute{N})$%
\begin{equation*}
\breve{g}(\breve{P}W,\breve{P}U)=g(W,\breve{P}U)+g(W,U),
\end{equation*}%
from the above equation and from 
\begin{equation*}
\breve{P}^{2}=\breve{P}+I,
\end{equation*}%
expression we have 
\begin{equation}
\breve{g}(\breve{P}V,\xi )=g(V,\breve{P}\xi )=0,  \label{12}
\end{equation}

and%
\begin{equation}
\breve{g}(\breve{P}V,N)=\breve{g}(V,\breve{P}N)=0.  \label{13}
\end{equation}%
In this case $\breve{P}V$ in $RadT\acute{N}$ and $ltrT\acute{N}$ have no
components.

Similarly 
\begin{equation}
\breve{g}(\breve{P}V,W)=\breve{g}(V,\breve{P}W)=0,  \label{14}
\end{equation}%
\begin{equation}
\breve{g}(\breve{P}V,\breve{P}V_{1})=\breve{g}(V,\breve{P}^{2}V_{1})=\breve{g%
}(V,\breve{P}V_{1})+g(V,V_{1})=0,  \label{15}
\end{equation}%
obtained, where $W\in \Gamma (S(T\acute{N}))$, $\breve{P}V_{1}\in \Gamma (%
\breve{P}S(T\acute{N})),$ $V_{1}\in \Gamma (S(T\acute{N}))$. From here it is
seen that there is no $\breve{P}V$ in $S(T\acute{N})$ and no components in $%
\breve{P}(S(T\acute{N}))$. From the equations (\ref{12}),(\ref{13}),(\ref{14}%
) and (\ref{15}), the result is that $\mu $ distribution is invariant.
\end{proof}

\begin{proposition}
\label{birbir}Let $\breve{N}$ be Golden semi-Riemannian manifold. In this
case, there is no 1-lightlike transversal submanifold of $\breve{N}$
manifold.
\end{proposition}

\begin{proof}
Let us assume that $\acute{N}$ is a 1-lightlike transversal submanifold . In
this case 
\begin{equation*}
RadT\acute{N}=Sp\{\xi \}
\end{equation*}%
and%
\begin{equation*}
ltrT\acute{N}=Sp\{N\}.
\end{equation*}%
Thus from equation below 
\begin{equation*}
\breve{g}(\breve{P}W,\breve{P}U)=\breve{g}(W,\breve{P}U)+g(W,U),
\end{equation*}%
\begin{equation}
\breve{g}\breve{P}\xi ,\xi )=g(\xi ,\breve{P}\xi )=g(\breve{P}\xi ,\breve{P}%
\xi )-g(\xi ,\xi )=0,  \label{16}
\end{equation}%
obtained. On the other side from 
\begin{equation*}
\breve{P}(RadT\acute{N})=ltrT\acute{N},
\end{equation*}%
equation%
\begin{equation*}
\breve{P}\xi =N\in \Gamma (ltrT\acute{N}),
\end{equation*}%
should be, where we get 
\begin{equation*}
\breve{g}(\xi ,\breve{P}\xi )=g(\xi ,N)=1.
\end{equation*}%
This contrary to (\ref{16}). In that case, our acceptance is false; our
hypothesis is true.
\end{proof}

Let $\acute{N}$ be , transversal lightlike submanifold of Golden
semi-Riemannian $\breve{N}$ manifold. Definition \ref{bir} and Proposition %
\ref{birbir} give following results.

1.) $\dim(RadT\acute{N})\geq2$

2.) 3-dimensional transversal lightlike submanifold is 2-lightlike.

Let $\breve{N}$ be golden semi-Riemannian manifold and $\acute{N}$ be
transversal lightlike submanifold of $\breve{N}$. In this case the
projection transforms to $RadT\acute{N}$ and $S(T\acute{N})$ are
respectively $Q$ and $T$ for $W\in \Gamma (T\acute{N})$%
\begin{equation}
W=TW+QW,  \label{17}
\end{equation}%
where $TW\in \Gamma (S(T\acute{N}))$, $QW\in \Gamma (RadT\acute{N})$.

If $\breve{P}$ is applied to expression (\ref{17})%
\begin{equation}
\breve{P}W=\breve{P}TW+\breve{P}QW,  \label{18}
\end{equation}

we have. If $\breve{P}TW=KW$ and $\breve{P}QW=LW$ is written the (\ref{18})
equation becomes%
\begin{equation}
\breve{P}W=KW+LW.  \label{19}
\end{equation}%
Where $KW\in \Gamma (S(T\acute{N})$ and $LW$ $\in \Gamma (ltrT\acute{N})$.

Also for $V\in\Gamma(S(T\acute{N})$%
\begin{equation}
\breve{P}V=BV+CV  \label{20}
\end{equation}

obtained. Where $BV\in\Gamma(\breve{P}S(T\acute{N}))\oplus\Gamma(S(T\acute {N%
}))$ and $CV\in\Gamma(\mu)$.

Let $\breve{N}$ be locally Golden semi-Riemannian manifold and $\acute{N}$
be transversal lightlike submanifold of $\breve{N}$. For $\forall W,U\in
\Gamma (T\acute{N})$ since 
\begin{equation*}
\breve{\nabla}_{U}\breve{P}W=\breve{P}\breve{\nabla}_{U}W
\end{equation*}%
and (\ref{19}) and 
\begin{equation*}
\breve{\nabla}_{U}W=\nabla _{U}W+h^{l}(U,W)+h^{s}(U,W),
\end{equation*}%
\begin{equation*}
\breve{\nabla}_{U}(KW+LW)=\breve{P}\breve{\nabla}_{U}W,
\end{equation*}%
\begin{equation*}
\breve{\nabla}_{U}KW+\breve{\nabla}_{U}LW=\breve{P}\breve{\nabla}_{U}W,
\end{equation*}%
where $KW\in \Gamma (S(T\acute{N}))$, $LW\in \Gamma (ltrT\acute{N})$.

\begin{equation*}
\binom{-A_{KW}U+\nabla _{U}^{s}KW+D^{l}(U,KW)}{-A_{LW}U+\nabla
_{U}^{l}LW+D^{s}(U,LW)}=\breve{P}\nabla _{U}W+\breve{P}h^{l}(U,W)+\breve{P}%
h^{s}(U,W),
\end{equation*}%
\begin{equation*}
\binom{-A_{KW}U+\nabla _{U}^{s}KW+D^{l}(U,KW)}{-A_{LW}U+\nabla
_{U}^{l}LW+D^{s}(U,LW)}=\left( 
\begin{array}{c}
K\nabla _{U}W+L\nabla _{U}W+\breve{P}h^{l}(U,W) \\ 
+Bh^{s}(U,W)+Ch^{s}(U,W)%
\end{array}%
\right)
\end{equation*}%
obtained.

Let $\breve{P}S(T\acute{N})\subset S(T\acute{N})$ and $S(T\acute{N})$ denote
the parts of $Bh^{s}(U,W)$ with $M_{2}h^{s}(U,W)$ and $M_{1}h^{s}(U,W)$
respectively; also if called 
\begin{equation*}
\breve{P}\xi _{1}=N_{1},
\end{equation*}%
additionally if $\breve{P}$ is applied to this expression since 
\begin{align*}
\breve{P}^{2}\xi _{1}& =\breve{P}N_{1}, \\
\breve{P}\xi _{1}+\xi _{1}& =\breve{P}N_{1},
\end{align*}%
we have 
\begin{equation*}
\breve{P}h^{l}(U,W)=K_{1}\breve{P}h^{l}(U,W)+K_{2}\breve{P}h^{l}(U,W).
\end{equation*}%
where $K_{1}$ and $K_{2}$ respectively projection morphism of $\breve{P}%
h^{l}(U,W)$ according to $ltrT\acute{N}$ and $RadT\acute{N}$.

Then (\ref{21}) is 
\begin{equation*}
\binom{-A_{KW}U+\nabla _{U}^{s}KW+D^{l}(U,KW)}{-A_{LW}U+\nabla
_{U}^{l}LW+D^{s}(U,LW)}=\left( 
\begin{array}{c}
K\nabla _{U}W+L\nabla _{U}W \\ 
+K_{1}\breve{P}h^{l}(U,W)+K_{2}\breve{P}h^{l}(U,W) \\ 
+M_{2}h^{s}(U,W)+M_{1}h^{s}(U,W)+Ch^{s}(U,W)%
\end{array}%
\right) 
\end{equation*}%
is obtained. If the tangent and transversal components are mutually equalized%
\begin{equation}
-A_{KW}U-A_{LW}U=K_{2}\breve{P}h^{l}(U,W)+M_{1}h^{s}(U,W)  \label{22}
\end{equation}%
\begin{equation}
\nabla _{U}^{s}KW+D^{s}(U,LW)=K\nabla _{U}W+M_{2}h^{s}(U,W)+Ch^{s}(U,W)
\label{23}
\end{equation}%
\begin{equation}
D^{l}(U,KW)+\nabla _{U}^{l}LW=L\nabla _{U}W+K_{1}\breve{P}h^{l}(U,W)
\label{24}
\end{equation}%
equations are obtained. Now, let's examine the integrability of
distributions on transversal lightlike submanifolds.

\begin{theorem}
Let $\acute{N}$ be, transversal lightlike submanifold of locally Golden
semi-Riemannian $\breve{N}$ manifold. In this case necessary and sufficient
condition for the distribution $RadT\acute{N}$ integrable is that for $%
\forall W,U\in \Gamma (RadT\acute{N})$%
\begin{equation*}
D^{s}(U,LW)=D^{s}(W,LU).
\end{equation*}
\end{theorem}

\begin{proof}
If the roles of $W$ and $U$ vector field are changed in equation (\ref{23})

\begin{equation*}
\nabla _{U}^{s}KW+D^{s}(U,LW)=K\nabla _{U}W+M_{2}h^{s}(U,W)+Ch^{s}(U,W),
\end{equation*}%
\begin{equation}
\nabla _{W}^{s}KU+D^{s}(W,LU)=K\nabla _{W}U+M_{2}h^{s}(W,U)+Ch^{s}(U,W),
\label{25}
\end{equation}%
obtained. From (\ref{23}) and (\ref{25}) equations because of $h^{s}$ is
symmetric for $U,W\in \Gamma (RadT\acute{N}),$%
\begin{equation*}
\Longrightarrow \breve{P}W=KW+LW,
\end{equation*}%
$KW=0$,

\begin{equation*}
\nabla _{U}^{s}KW=\nabla _{W}^{s}KU=0
\end{equation*}%
is obtained. Then we have 
\begin{equation}
D^{s}(U,LW)-D^{s}(W,LU)=K[U,W].  \label{26}
\end{equation}%
Thus the proof is complete.
\end{proof}

\begin{theorem}
Let $\acute{N}$ be , transversal lightlike submanifold of locally Golden
semi-Riemannian $\breve{N}$ manifold. In this case, screen distribution is
necessary and sufficient condition for the definition of totally geodesic
foliation on $\acute{N}$ is for $W,U\in \Gamma ($ $S(T\acute{N}))$, $N\in
\Gamma (ltrT\acute{N})$%
\begin{equation*}
h^{\ast }(W,U)=0
\end{equation*}%
and $A_{\breve{P}U}W$ in $RadT\acute{N}$ is the absence of component.
\end{theorem}

\begin{proof}
$S(T\acute{N})$ distribution is necessary and sufficient condition for the
definition of totally geodesic foliation 
\begin{equation*}
g(\nabla _{W}U,N)=0,
\end{equation*}%
for $W,U\in \Gamma (S(T\acute{N})$ $)$, $N\in \Gamma (ltrT\acute{N})$.%
\begin{align*}
g(\nabla _{W}U,N)& =\breve{g}(\breve{\nabla}_{W}U-h^{l}(W,U)-h^{s}(W,U),N),
\\
& =\breve{g}(\breve{\nabla}_{W}U,N), \\
& =\breve{g}(\breve{P}(\breve{\nabla}_{W}U,\breve{P}N)-\breve{g}(\breve{%
\nabla}_{W}U,\breve{P}N), \\
& =\breve{g}(\breve{\nabla}_{W}\breve{P}U,\breve{P}N)-g(\breve{\nabla}_{W}U,%
\breve{P}N),\text{\quad }\breve{P}U\in \Gamma (S(T\acute{N})) \\
& =\breve{g}(-A_{\breve{P}U}W+\nabla _{W}^{s}\breve{P}U+D^{l}(W,\breve{P}U),%
\breve{P}N), \\
& -\breve{g}(\nabla _{W}^{\ast }U+h^{\ast }(W,U)+h^{l}(W,U)+h^{s}(W,U),%
\breve{P}N).
\end{align*}%
Since $\breve{P}N\in \Gamma (RadT\acute{N})\oplus \Gamma (ltrT\acute{N})$

\begin{align*}
g(\nabla _{W}U,N)& =-\breve{g}(A_{\breve{P}U}W,\breve{P}N)-\breve{g}(h^{\ast
}(W,U),\breve{P}N), \\
0& =-\breve{g}(A_{\breve{P}U}W+h^{\ast }(W,U),\breve{P}N).
\end{align*}%
In this way we have $h^{\ast }(W,U)=0$ and there is no component $A_{\breve{P%
}U}W$ in $RadT\acute{N}$.
\end{proof}

\begin{theorem}
Let $\acute{N}$ be transversal lightlike submanifold of locally Golden
semi-Riemannian $\breve{N}$ manifold. In this case radical distribution is
necessary and sufficient condition for the definition totally geodesic
foliation on $\acute{N}$ for $\forall U,W\in \Gamma (RadT\acute{N})$, $Z\in
\Gamma (S(T\acute{N}))\ A_{\breve{P}Z}W$ in $RadT\acute{N}$ is absence of
component; namely 
\begin{equation*}
K_{2}\breve{P}h^{l}(W,Z)=0,
\end{equation*}%
or

\begin{equation*}
-A_{\breve{P}Z}W=M_{1}h^{s}(W,Z).
\end{equation*}
\end{theorem}

\begin{proof}
For $\forall U,W\in\Gamma(RadT\acute{N})$, $Z\in\Gamma(S(T\acute{N}))$

$RadT\acute{N}$ distribution is necessary and sufficient condition for the
definition totally geodesic foliation 
\begin{equation*}
g(\nabla _{W}U,Z)=0.
\end{equation*}%
From here 
\begin{equation*}
g(\nabla _{W}U,Z)=g(\breve{\nabla}_{W}U,Z).
\end{equation*}%
Since $\breve{\nabla}$ is a metric connection 
\begin{equation*}
g(\nabla _{W}U,Z)=W\breve{g}(U,Z)-\breve{g}(U,\breve{\nabla}_{W}Z)
\end{equation*}%
is obtained. If $\breve{g}(\breve{P}W,\breve{P}U)=\breve{g}(W,\breve{P}U)+%
\breve{g}(W,U)$ is used here. 
\begin{align*}
g(\nabla _{W}U,Z)& =-\breve{g}(\breve{P}U,\breve{P}\breve{\nabla}_{W}Z)+%
\breve{g}(U,\breve{P}\breve{\nabla}_{W}Z)\quad \breve{P}Z\in \Gamma (S(T%
\acute{N}), \\
& =-\breve{g}(\breve{P}U,\breve{\nabla}_{W}\breve{P}Z)+g(\breve{P}U,\breve{%
\nabla}_{W}Z)\quad Z\in \Gamma (S(T\acute{N}))\,W\in \Gamma (RadT\acute{N}),
\\
& =-\breve{g}(\breve{P}U,A_{\breve{P}Z}W+\nabla _{W}^{s}\breve{P}Z+D^{l}(W,%
\breve{P}Z)), \\
+& \breve{g}(\breve{P}U,\nabla _{W}^{\ast }Z+h^{\ast
}(W,Z)+h^{l}(W,Z)+h^{s}(W,Z)\quad \breve{P}U\in \Gamma (ltrTN), \\
& =\breve{g}(\breve{P}U,A_{\breve{P}Z}W)+\breve{g}(h^{\ast }(W,Z),\breve{P}%
U), \\
& =\breve{g}(A_{\breve{P}Z}W+h^{\ast }(W,Z),\breve{P}U)\text{\quad }W\in
\Gamma (RadT\acute{N}), \\
0& =\breve{g}(A_{\breve{P}Z}W,\breve{P}U).
\end{align*}%
Since $\breve{P}U\in \Gamma (ltrT\acute{N})$ this is the result that $A_{%
\breve{P}Z}W$ is not compounded in $RadT\acute{N}$. Or from (\ref{22})%
\begin{equation*}
-A_{KW}U-A_{LW}U=K_{2}\breve{P}h^{l}(U,W)+M_{1}h^{s}(U,W),
\end{equation*}%
using the above equation 
\begin{equation*}
-A_{\breve{P}W}U=K_{2}\breve{P}h^{l}(U,W)+M_{1}h^{s}(U,W)
\end{equation*}%
is obtained. If the roles of $U$ and $W$ changed 
\begin{equation*}
-A_{\breve{P}U}W=K_{2}\breve{P}h^{l}(W,U)+M_{1}h^{s}(W,U).
\end{equation*}%
For $U=Z$%
\begin{equation*}
-A_{\breve{P}Z}W=K_{2}\breve{P}h^{l}(W,Z)+M_{1}h^{s}(W,Z),
\end{equation*}%
from above equation 
\begin{equation*}
K_{2}\breve{P}h^{l}(W,Z)=0
\end{equation*}%
is obtained.
\end{proof}

\begin{theorem}
Let $\acute{N}$ be transversal lightlike submanifold of locally Golden
semi-Riemannian $\breve{N}$ manifold. In this case, it is necessary and
sufficient condition for the induced connection $\nabla $ is metric
connection is 
\begin{equation*}
Q_{1}\breve{P}D^{s}(W,\xi )=M_{1}\breve{P}h^{s}(W,\xi )\text{ \ \ for }W\in
\Gamma (T\acute{N})\text{, }\xi \in \Gamma (RadT\acute{N}),
\end{equation*}%
where $Q_{1}$ and $M_{1}$ are projection morphism on $\breve{P}D^{s}(W,S)$
and $\breve{P}h^{s}(W,S)$ on $S(T\acute{N})$ respectively.
\end{theorem}

\begin{proof}
Since 
\begin{align*}
\breve{\nabla}_{W}\breve{P}\xi -\breve{P}\breve{\nabla}_{W}\xi & =0, \\
\breve{\nabla}_{W}\breve{P}\xi & =\breve{P}\breve{\nabla}_{W}\xi ,
\end{align*}%
we obtain%
\begin{equation*}
-A_{\breve{P}\xi }W+\nabla _{W}^{l}\breve{P}\xi +D^{s}(W,\breve{P}\xi )=%
\breve{P}(\nabla _{W}\xi +h^{l}(W,\xi )+h^{s}(W,\xi )).
\end{equation*}%
If $\breve{P}$ is applied above equation 
\begin{equation*}
-\breve{P}A_{\breve{P}\xi }W+\breve{P}\nabla _{W}^{l}\breve{P}\xi +\breve{P}%
D^{s}(W,\breve{P}\xi )=\breve{P}^{2}\nabla _{W}\xi +\breve{P}^{2}h^{l}(W,\xi
)+\breve{P}^{2}h^{s}(W,\xi )
\end{equation*}%
is obtained. If 
\begin{equation*}
\breve{P}^{2}=\breve{P}+I
\end{equation*}%
and 
\begin{equation*}
\breve{P}W=KW+LW\text{, \ for }W\in \Gamma (T\acute{N}),
\end{equation*}%
\begin{equation*}
\breve{P}V=BV+CV\text{, for }V\in \Gamma (S(T\acute{N})),
\end{equation*}%
is taken into account 
\begin{equation*}
\left( 
\begin{array}{c}
-KA_{\breve{P}\xi }W-LA_{\breve{P}\xi }W \\ 
+\breve{P}\nabla _{W}^{l}\breve{P}\xi +\breve{P}D^{s}(W,\breve{P}\xi )%
\end{array}%
\right) =\binom{\breve{P}\breve{\nabla}_{W}\xi +\nabla _{W}\xi +\breve{P}%
h^{l}(W,\xi )+h^{l}(W,\xi )}{\breve{P}h^{s}(W,\xi )+h^{s}(W,\xi )}
\end{equation*}%
is obtained. If the tangent parts are equalized 
\begin{equation*}
\breve{P}\nabla _{W}^{l}\breve{P}\xi =T_{1}\breve{P}\nabla _{W}^{l}\breve{P}%
\xi +T_{2}\nabla _{W}^{l}\breve{P}\xi ,
\end{equation*}%
where $T_{1}$ and $T_{2}$ are projection morphisms of $\breve{P}\nabla
_{W}^{l}\breve{P}\xi $ on $RadT\acute{N}$ and $ltrT\acute{N}$ respectively.

If $Q_{1}$ and $Q_{2}$ are projection morphisms of $\breve{P}D^{s}(W,\breve{P%
}\xi )$ on $S(T\acute{N})$ and $S(T\acute{N})$ respectively. From 
\begin{equation*}
\breve{P}D^{s}(W,\breve{P}\xi )=Q_{1}\breve{P}D^{s}(W,\xi )+Q_{2}\breve{P}%
D^{s}(W,\xi ),
\end{equation*}%
we have 
\begin{equation*}
T_{1}\breve{P}\nabla _{W}^{l}\breve{P}\xi +Q_{1}\breve{P}D^{s}(W,\xi
)=\nabla _{W}\xi +M_{1}\breve{P}h^{s}(W,\xi )+K_{2}\breve{P}h^{l}(W,\xi ).
\end{equation*}%
From here 
\begin{equation*}
\nabla _{W}\xi =T_{1}\breve{P}\nabla _{W}^{l}\breve{P}\xi +Q_{1}\breve{P}%
D^{s}(W,\xi )-M_{1}\breve{P}h^{s}(W,\xi )-K_{2}\breve{P}h^{l}(W,\xi )
\end{equation*}%
is obtained. From this equation for $\nabla _{W}\xi \in \Gamma (RadT\acute{N}%
)$ necessary and sufficient condition is 
\begin{equation*}
Q_{1}\breve{P}D^{s}(W,\xi )-M_{1}\breve{P}h^{s}(W,\xi )=0,
\end{equation*}%
namely 
\begin{equation*}
Q_{1}\breve{P}D^{s}(W,\xi )=M_{1}\breve{P}h^{s}(W,\xi ).
\end{equation*}
\end{proof}

\end{document}